\documentclass[reqno]{amsart}
\usepackage{amssymb}
\usepackage{amsmath}
\usepackage{enumerate}
\usepackage{srcltx}
\usepackage[mathscr]{euscript}
\usepackage{times}

\usepackage{tikz}

\makeatletter
\@addtoreset{equation}{section}
\makeatother

\newtheorem{theorem}{Theorem}[section]
\newtheorem{lemma}[theorem]{Lemma}
\newtheorem{proposition}[theorem]{Proposition}

\newtheorem{definition}[theorem]{Definition}

\newcounter{as}

\newcommand{\mc}[1]{{\mathcal #1}}
\newcommand{\mf}[1]{{\mathfrak #1}}
\newcommand{\mb}[1]{{\mathbf #1}}
\newcommand{\bb}[1]{{\mathbb #1}}
\newcommand{\bs}[1]{{\boldsymbol #1}}
\newcommand{\ms}[1]{{\mathscr #1}}

\newcommand{\<}{\langle}
\renewcommand{\>}{\rangle}
\renewcommand{\Cap}{{\rm cap}}

\title[Variational formulae for the capacity]
{Variational formulae for the capacity induced by second-order
  elliptic differential operators}
  
\author{C. Landim}

\address{\noindent IMPA, Estrada Dona Castorina 110, CEP 22460 Rio de
  Janeiro, Brasil and CNRS UMR 6085, Universit\'e de Rouen, Avenue de
  l'Universit\'e, BP.12, Technop\^ole du Madril\-let, F76801
  Saint-\'Etienne-du-Rouvray, France.  \newline e-mail: \rm
  \texttt{landim@impa.br} }

\begin{document}

\begin{abstract}
We review recent progress in potential theory of second-order
  elliptic operators and on the metastable behavior of Markov
  processes.
\end{abstract}

\maketitle

There has been many recent progress in the potential theory of
non-reversible Markov processes. We review in this article some of
these advances. In Section \ref{sec01}, we present a brief historical
overview of potential theory and we introduce the main notions which
will appear throughout the article. In Section \ref{sec02}, we present
two variational formulae for the capacity between two sets induced by
second-order elliptic operators non necessarily self-adjoint. In the
following three sections we present applications of these results. In
Section \ref{sec04}, we discuss recurrence of Markov processes; in
Section \ref{sec05}, we present a sharp estimate for the transition
time between two wells in a dynamical system randomly perturbed; and
in Section \ref{sec03}, we prove the metastable behavior of this
process.

\section{Potential theory}
\label{sec01}

We present in this section a brief historical introduction to the
Dirichlet principle. The interested reader will find in Kellogg's book
\cite{Kel} a full account and references.

\noindent{\bf From Newton's law of universal gravitation to Laplace's equation.}
In 1687, Newton enunciated the Law of universal gravitation which
states that ``every particle of matter in the universe attracts every
other particle with a force whose direction is that of the line
joining the two, and whose magnitude is directly as the product of
their masses, and inversely as the square of their distance from each
other''. The magnitude $F$ of the force between two particules, one of
mass $m_1$ situated at $x\in\bb R^3$ and one of mass $m_2$ situated at
$y\in \bb R^3$ is thus given by
\begin{equation}
\label{01}
F \;=\; \kappa\, \frac{m_1 \, m_2}{\Vert x-y\Vert^2}\;, 
\end{equation}
where $\Vert (z_1,z_2,z_3)\Vert = \sqrt{z^2_1 + z^2_2 + z^2_3}$ stands
for the Euclidean distance, and $\kappa$ for a constant which depends
only on the units used. In order to avoid the constant $\kappa$ we
choose henceforth the unit of force so that $\kappa = 1$.

Once Newton's gravitation law has been formulated, it is natural to
calculate the force exerted on a particle of unit mass by different
types of bodies.  Consider a body $\mf B$ occupying a portion $\Omega$
of the space $\bb R^3$.  Assume that its density $\rho$ at each point
$z\in \Omega$ is well defined and that it is continuous and bounded as
a function of $z$. By density at $z$ we mean the limit of the ratio
between the mass of a portion containing $z$ and the volume of the
portion, as the volume of the portion vanishes.  By \eqref{01}, the
force at a point $x\in\bb R^3$ is given by
\begin{equation}
\label{02}
F(x) \;=\; \int_{\Omega} \frac{z-x}{\Vert z-x\Vert^3} \, \rho(z)\,
dz \;.
\end{equation}
Note that the force is well defined in $\Omega$ because $x\mapsto
\Vert x\Vert^{-2}$ is integrable in a neighborhood of the origin and
we assumed the density $\rho$ to be bounded.  Equation \eqref{02}
defines, therefore, a vector field $F = (F_1, F_2, F_3): \bb R^3 \to
\bb R^3$.

The force field $F$ introduced in \eqref{02} turns out to be
\emph{divergence free} in $\Omega^c$:
\begin{equation}
\label{03}
(\nabla \cdot F )(x) \;:=\; \sum_{j=1}^3 (\partial_{x_j} F_j)(x) \;=\;
0\;, \quad x\in \Omega^c\;,
\end{equation}
where $\partial_{x_j}$ represents the partial derivative with respect
to $x_j$. It is also \emph{conservative}: Fix a point $x \in \bb R$, and
let $\gamma: [0,1]\to\bb R^3$ be a smooth, closed path such that
$\gamma(0) =\gamma(1) =x$. The integral of the force field along
the cycle $\gamma$ is given by
\begin{equation*}
\int F \cdot d\gamma \;:=\; \sum_{j=1}^3
\int F_j(\gamma) \, d \gamma_j \;=\;
\sum_{j=1}^3 \int_0^1 F_j(\gamma(t)) \, \gamma'_j (t)\,
dt \;=\;0\;.
\end{equation*}

As the force field is conservative and the space is simply connected
[any two paths with the same endpoints can be continuously deformed
one into the other], we may associate a potential $V:\bb R^3 \to \bb
R$ to the vector field $F$. Fix a point $x_0\in\bb R^3$ and a constant
$C_0$, and let
\begin{equation}
\label{05}
V(x) \;=\; C_0 \;+\; \int F \cdot d\gamma \;,
\end{equation}
where $\gamma$ is a continuous path from $x_0$ to $x$. The potential
$V$ is well defined because the force field is conservative, and it is
unique up to an additive constant. By requiring it to vanish at
infinity, it becomes
\begin{equation}
\label{1-1}
V(x) \;=\; -\, \int_{\bb R^3} \frac{1}{\Vert z-x\Vert} \, \rho(z)\,
dz \;,
\end{equation}
and it is called the Newton potential of the measure $\rho(z)\, dz$.
Moreover, if we represent by $\nabla V$ the gradient of $V$, $\nabla V
= (\partial_{x_1} V \,,\, \partial_{x_2} V\,,\, \partial_{x_3} V)$,
\begin{equation}
\label{07}
\nabla  V \;=\; F\;.
\end{equation}
Hence, since the force field is divergence-free [equation \eqref{03}],
\begin{equation}
\label{04}
\Delta V \;:=\; \nabla \cdot \nabla V \;=\; 0 \quad\text{on } \Omega^c
\;,
\end{equation}
which is known as Laplace's differential equation.

This last identity provides an alternative way to compute the force
field induced by a body whose density is unknown. Let $\mf B$ be a
body occupying a portion $\Omega$ of the space $\bb R^3$. Assume that
$\Omega^c$ is a domain [open and connected] which has a smooth, simply
connected boundary, denoted by $\partial \Omega$. Assume, further,
that the force field $F$ exerted by the body $\mf B$ can be measured
at the boundary of $\Omega$. Fix a point $x_0\in \partial\Omega$,
set $V(x_0)=0$, and extend the definition of $V$ to $\partial\,
\Omega$ through equation \eqref{05}. By \eqref{04}, the potential $V$
solves the equation
\begin{equation}
\label{06}
\begin{cases}
(\Delta W) (x) \;=\; 0 \;, \;\; x\in \Omega^c  \;, \\
W(x) \;=\; V(x) \;, \;\; x\in \partial\Omega \;.
\end{cases}
\end{equation}
To derive $F$, it remains to solve the linear equation \eqref{06} and
to retrieve $F$ from $V$ by \eqref{07}.

The problem of proving the existence of a function satisfying
\eqref{06} or of finding it when it exists is known as the Dirichlet
problem, or the first boundary problem of potential theory.
 
\smallskip\noindent{\bf Dirichlet's principle.}  In 1850, Dirichlet
proposed the following argument to prove the existence of a solution
to \eqref{06}. It is simpler to present it in the context of masses
distributed along surfaces.  If mass points disturb, on may think in
terms of charges since, according to Coulomb's law, two point charges
exert forces on each other which are given by Newton's law with the
word mass replaced by charge, except that charges may attract or repel
each other. 

Consider a bounded domain $\Omega$ whose boundary, represented by
$\partial \Omega$, is smooth. Let $\zeta$ be a surface density on
$\partial \Omega$. By \eqref{1-1}, the potential associated to this
mass distribution is given by
\begin{equation*}
V(x) \;=\; -\, \int_{\partial \Omega} \frac{1}{\Vert z-x\Vert} \,
\zeta (z)\, \sigma(dz)\;,
\end{equation*}
where $\sigma(dz)$ stands for the surface measure. The surface
density can be recovered from the potential. By Theorem VI of Chapter
VI in \cite{Kel},
\begin{equation}
\label{09}
\frac{\partial V}{\partial n_+} (x)  \;-\; 
\frac{\partial V}{\partial n_-} (x) 
\;=\; -\, 4\pi\, \zeta(x)\;, 
\end{equation}
where $n_+$, resp. $n_-$, represents the outward, resp. inward,
pointing unit normal vector to $\partial \Omega$. 

Denote by $E(\zeta)$ the potential energy of the mass distribution
$\zeta$. It corresponds to the total work needed to assemble the
distribution from a state of infinite dispersion, and it is given by
\begin{equation*}
E(\zeta) \;=\; 
\frac 12\, \int_{\partial \Omega} V (x)\, \zeta (x)\, \sigma(dx)\;.
\end{equation*} 
Since, by \eqref{09}, the surface density can be expressed in terms of
the potential, we may consider the energy as a function of the
potential. After this replacement, as the potential satisfies
Laplace's equation \eqref{04} on $(\partial \Omega)^c$, applying the
divergence theorem, we obtain that
\begin{equation*}
E(V) \;=\; \frac 1{8\pi}\, \int_{\bb R^3} 
\Vert \nabla V (x) \Vert^2 \, dx \;.
\end{equation*}

It is a principle of physics that equilibrium is characterized by the
least potential energy consistent with the constraints of the
problem. Thus, to prove the existence of a solution of the
differential equation \eqref{06}, Dirichlet proposed to consider the
variational problem
\begin{equation*}
\inf_f \int_{\bb R^3} \Vert (\nabla f)(x) \Vert^2 \, dx\;,
\end{equation*}
where the infimum is carried over all smooth functions $f: \bb R^3 \to
\bb R$ such that $f=V$ on $\partial\Omega$.

Mathematicians objected to the argument at an early date, pointing
that the infimum might not be attained at an element of the class of
functions considered. Weierstrass gave an example showing that the
principle was false, and in 1899, Hilbert provided conditions on the
surface, the boundary values and the class of functions $f$ admitted,
for which the Dirichlet principel could be proved.

\smallskip\noindent{\bf Condenser capacity.}  In electrostatics, the
capacity of an isolated conductor is the the total charge the
conductor can hold at a given potential.

Let $\Omega_1 \subset \Omega_2$ be bounded domains with smooth
boundaries represented by $\Sigma_1$, $\Sigma_2$, respectively. Assume
that the closure of $\Omega_1$ is contained in $\Omega_2$.  Consider
the potential which is equal to $1$ at $\Omega_1$, $0$ at
$\Omega^c_2$, and which satisfies Laplace's equation on $\ms R =
\Omega_2 \setminus \overline{\Omega}_1$. Since $V$ satisfies Laplace's
equation on $(\Sigma_1 \cup \Sigma_2)^c$, this potential can be
obtained from a surface distribution concentrated on $\Sigma_1 \cup
\Sigma_2$. The total mass [charge] on $\Sigma_1$ is given by
\begin{equation*}
\int_{\Sigma_1} \zeta(x)\, \sigma(dx) \;=\;
-\, \frac 1{4\pi} \int_{\Sigma_1} 
\frac{\partial V}{\partial n_+} (x) \, \sigma(dx)\;,
\end{equation*}
where the last identity follows from \eqref{09} and from the fact that
the inward derivative vanishes because $V$ is constant in $\Omega_1$.
The condenser capacity of $\Omega_1$ relative to $\Omega_2$ is given
by
\begin{equation}
\label{1-01}
\Cap (\Omega_1, \Omega^c_2) \;=\; 
-\, \frac 1{4\pi} \int_{\Sigma_1} 
\frac{\partial V}{\partial n_+} (x) \, \sigma(dx)\;,
\end{equation}
The measure $\nu = -\, (1/4\pi)\, (\partial V/\partial n_+) (x) \,
\sigma(dx)$ on $\Sigma_1$ is called the \emph{harmonic measure}.
The capacity of $\Omega_1$ is obtained by letting $\Omega_2$ increase
to $\bb R^d$.

As the potential $V$ is equal to $1$ on $\Omega_1$ and $0$ at
$\Sigma_2$, we may insert $V(x)$ in the previous integral, add the
integral of the same expression over $\Sigma_2$, and then use the
divergence theorem and the fact that $V$ is harmonic on the annulus
$\ms R = \Omega_2\setminus\Omega_1$ to conclude that the previous
expression is equal to
\begin{equation*}
\frac 1{4\pi} \int_{\ms R}
\Vert (\nabla V)(x)\Vert^2 \, dx \;=\; 
\frac 1{4\pi} \inf_f \int_{\bb R^3} 
\Vert (\nabla f)(x)\Vert^2 \, dx\;,
\end{equation*}
where the infimum is carried over all smooth functions $f$ such that
$f=1$ on $\Omega_1$ and $f=0$ on $\Omega_2^c$. This latter formula
provides a variational formula for the capacity defined by
\eqref{1-01}, called the Dirichlet principle. 

In the next section, we present two variational formulae for the
capacity induced by a second-order elliptic operator which is not
self-adjoint with respect to the stationary state [as it is the case
of the Laplace operator with respect to the Lebesgue measure]. We then
present some applications of the formulae.

\section{Dirichlet and Thomson principles}
\label{sec02}

In this section, we extend the notion of capacity to the context of
general second order differential operators not necessarily
self-adjoint. We then provide two variational formulae for the
capacity, the so-called Dirichlet and Thomson principles. We will not
be precise on the smoothness conditions of the functions and of the
boundary of the sets. The interested reader will find in the
references rigorous statements.

To avoid integrability conditions at infinity, we state the Dirichlet
and the Thomson principles on a finite cube with periodic boundary
conditions.  Fix $d\ge 1$, and denote by $\bb T^d = [0,1)^d$ the
$d$-dimensional torus of length $1$.  Denote by {\color{blue}$a(x)$} a
uniformly positive-definite matrix whose entries $a_{i,j}$ are smooth
functions: There exist $c_0>0$ such that for all $x\in \bb T^d$,
$\xi\in \bb R^d$,
\begin{equation}
\label{19}
\xi \cdot a (x) \, \xi \;\ge\; c_0 \, \Vert\xi \Vert^2 \;,
\end{equation}
where {\color{blue} $\eta\cdot \xi$} represents the scalar product in
$\bb R^d$.

\smallskip\noindent{\bf Generator.}
Denote by $\mc L$ the generator given by
\begin{equation}
\label{2-12}
\mc L f\;=\; \nabla \cdot (a \, \nabla f) \;+\; b \cdot \nabla f \;, 
\end{equation}
where $b: \bb T^d \to \bb R^d$ is a smooth vector field. By modifying
the drift $b$ we could assume the matrix $a$ to be symmetric. We will
not assume this condition for reasons which will become clear in the
next sections. There exists a unique Borel probability measure such
that $\mu \mc L=0$. This measure is absolutely continuous, $\mu(dx) =
m(x) dx$, where $m$ is the unique solution to
\begin{equation}
\label{2-16}
\nabla \cdot (a^\dagger \, \nabla m) \;-\; \nabla \cdot (b\, m) \;=\;
0 \;, 
\end{equation}
where {\color{blue}$a^\dagger$} stands for the transpose of $a$. For
existence, uniqueness and regularity conditions of solutions of
elliptic equations, we refer to \cite{gt}. Let $V(x) = - \log m(x)$,
so that {\color{blue}$m(x) = e^{-V(x)}$}.

Throughout this section $A$, $B$ represent two closed, disjoint
subsets of $\bb T^d$ 
\begin{equation}
\label{2-15}
\text{which are the closure of open sets with smooth boundaries.}
\end{equation}
For such a set $A$, denote by {\color{blue}$\mu_A (dx)$} the measure
$m(x) \sigma(dx)$ on $\partial A$, where {\color{blue}$\sigma(dx)$}
represents the surface measure. Hence, for every smooth vector field
$\varphi$,
\begin{equation*}
\oint_{\partial A} \varphi (x) \cdot n_A (x) \, \mu_A(dx)
\;=\; \oint_{\partial A} \varphi  (x) \cdot n_A (x) \, e^{-V(x)}
\, \sigma (d x)\;,
\end{equation*}
where {\color{blue}$n_A$} represents the {\color{blue}inward} normal
vector to $\partial A$.

We may rewrite the generator $\mc L$ introduced in \eqref{2-12} as
\begin{equation*}
\mc L f \;=\;  e^V \nabla \cdot \big( e^{-V} a \nabla f\big)
\;+\; c\cdot \nabla f\;,
\end{equation*}
where {\color{blue} $c=b + a^\dagger \nabla V$}. It follows from
\eqref{2-16} that
\begin{equation}
\label{2-07}
\nabla \cdot (e^{-V} c) \;=\; 0\;.
\end{equation}
This implies that the operator $c\cdot \nabla$ is skew-adjoint in
$L_2(\mu)$: for any smooth functions $f$, $g: \bb T^d\to \bb R$,
\begin{equation}
\label{2-01}
\int f \, c\cdot \nabla g \, d\mu
\;=\; -\, \int g\,  c\cdot \nabla f \, d\mu\;,
\end{equation}
and that for any open set $D$ with a smooth boundary,
\begin{equation}
\label{2-05}
\oint_{\partial D} c \cdot n_D  \, d\mu_D 
\;=\; \int_{\bb T^d\setminus D} e^V \, \nabla \cdot (e^{-V} c)\, d\mu 
\;=\; 0\;.
\end{equation}

In view of \eqref{2-01}, the adjoint of $\mc L$ in $L_2(\mu)$,
represented by $\mc L^*$, is given by
\begin{equation*}
\mc L^* f \;=\; e^V \, \nabla \cdot \big( e^{-V} a^\dagger \nabla f\big)
\;-\; c\cdot \nabla f\;,
\end{equation*}
while the symmetric part, denoted by $\mc L^s$, $\mc L^s = (1/2)(\mc L
+ \mc L^*)$, takes the form
\begin{equation}
\label{2-8}
\mc L^s f \;=\; e^V \, \nabla \cdot \big( e^{-V} a_s \nabla f\big) \;.
\end{equation}
where $a_s$ stands for the symmetrization of the matrix $a$:
{\color{blue}$a_s = (1/2)[a + a^\dagger]$}.

\smallskip\noindent{\bf Capacity.} Recall that $A$, $B$ are closed
sets satisfying \eqref{2-15}. Let
\begin{equation*}
\Omega \;=\; \bb T^d\setminus (A \cup B)\;.
\end{equation*}
Denote by {\color{blue}$h = h_{A,B}$}, resp. {\color{blue}$h^*=
  h^*_{A,B}$}, the unique solutions to the linear elliptic equations
\begin{equation}
\label{2-03}
\begin{cases}
\mc Lh=0 & \text{on $\Omega$,}
\\
h= \chi_A & \text{on $A\cup B$,}
\end{cases}
\qquad
\begin{cases}
\mc L^*h=0 & \text{on $\Omega$,}
\\
h^* = \chi_A & \text{on $A\cup B$,}
\end{cases}
\end{equation}
where $\chi_C$, $C\subset \bb T^d$, represents the indicator function
of the set $C$. The functions $h$, $h^*$ are called the {\color{blue}
  equilibrium potentials} between $A$ and $B$. A function $f$ such
that $(\mc L f)(x)=0$ is said to harmonic at $x$. If it is
{\color{blue} harmonic} at all points in some domain $\Omega$, it is
said to be harmonic in $\Omega$.

By analogy to the electrostatic definition \eqref{1-01} of the
capacity of a set, define the \emph{capacity} between the sets $A$,
$B$ of $\bb T^d$ as 
\begin{equation}
\label{2-09}
\Cap (A,B) \;:=\; \oint_{\partial A} a\, \nabla h \cdot n_A \,
d\mu_A\;, \quad 
\Cap^* (A,B) \;:=\; \oint_{\partial A} a^\dagger\, \nabla h^* \cdot n_A 
\, d\mu_A\;.
\end{equation}
Since $h=1$ at $\partial A$ and $h=0$ at $\partial B$, we may insert
$h$ in the integral and add the surface integral of the same
expression over $\partial B$. Applying then the divergence theorem, we
obtain that 
\begin{equation*}
\Cap (A,B) \;=\; \int_\Omega \nabla h \cdot a \, \nabla h \, d\mu \;+\;
\int_\Omega h \, e^V\, \nabla \cdot \big( e^{-V}\, a \, \nabla h \big)
\, d\mu \;.
\end{equation*}
As $\nabla h$ vanishes on $A\cup B$, we may extend the integrals to
$\bb T^d$. The integrand in the second term can be written as $h \,
[\mc Lh - c \cdot \nabla h]$. Since $c\cdot \nabla$ is skew-adjoint
and $h$ is harmonic on $(\partial A \cup \partial B)^c$, we conclude
that
\begin{equation}
\label{2-02}
\Cap(A,B) \;=\; \int \nabla h \cdot a \, \nabla h \, d\mu \;, \quad
\Cap^*(A,B) \;=\; \int \nabla h^*_{A,B} \cdot a^\dagger \, \nabla h^*_{A,B} \, d\mu \;.
\end{equation}
In the previous formulae, we may replace $a$, $a^\dagger$ by their
symmetric part $a_s$, and we may restrict the integrals to $\Omega$.

\begin{lemma}
\label{l2-01}
Let $A$, $B$ be two closed subsets satisfying the conditions
\eqref{2-15}. Then,
\begin{equation*}
\Cap(A,B) \;= \; \Cap(B,A) \;=\; \Cap^*(A,B) \;. 
\end{equation*}
Moreover, 
\begin{equation}
\label{2-14}
\Cap(A,B) \; =\; \oint_{\partial A} \left(a\, \nabla h + h\, c\right)\cdot n_A
\,d\mu_A \;. 
\end{equation}
\end{lemma}

\begin{proof}
It is clear that $\Cap(B,A)=\Cap(A,B)$ since $h_{B,A}=1-h_{A,B}$ as
$A\cap B=\varnothing$. The proof of \eqref{2-14} is similar to the one
which led from the definition of the capacity to \eqref{2-02}. One has
just to recall from \eqref{2-07} that $\nabla \cdot (e^{-V} c) =0$. 

We turn to the proof that $\Cap(A,B) = \Cap^*(A,B)$ It relies on the
claim that
\begin{align*}
\Cap(A,B) \; =\; \int_{\bb T^d} \big\{ \nabla h^* \cdot a\, \nabla h
- h^* \,c\cdot \nabla h \big\} \, d\mu
\;=\;\int_{\bb T^d} \big\{ \nabla h \cdot a^\dagger\, \nabla h^* 
+  h \,c\cdot \nabla h^*\big\} \, d\mu \;.
\end{align*}
To prove this claim, repeat the calculations carried out to derive
\eqref{2-14} to conclude that
\begin{equation*}
\int_{\bb T^d} \nabla h^* \cdot a\, \nabla h \, d\mu \;=\;
\int_{\Omega} \nabla h^* \cdot a\, \nabla h \, d\mu \;=\;
\oint_{\partial A}  a\, \nabla h\cdot n_A \, d\mu_{A}
\;+\; \int_{\Omega} h^*\, c\cdot \nabla h \, d\mu\;.
\end{equation*}
Since $\nabla h$ vanishes on $A\cup B$, we may carry the second
integral over $\bb T^d$. This proves the first identity of the claim
because the first term on the right hand side is equal to the capacity
between $A$ and $B$.

The same computation inverting the roles of $h$ and $h^*$ gives that
\begin{equation*}
\int_{\bb T^d} \nabla h \cdot a^\dagger\, \nabla h^* \, d\mu \;=\;
\Cap^*(A,B) \; -\; \int_{\bb T^d}  h\,c\cdot \nabla h^* \, d\mu \;.
\end{equation*}
Compare this identity with the previous one. The left-hand sides
coincide. As $c\cdot \nabla$ is skew-adjoint, the second terms on the
right-hand sides are also equal. Hence, $\Cap(A,B)=\Cap^*(A,B)$ because
the first term on the right-hand side of the penultimate equation is
$\Cap(A,B)$.  The previous identity together with the fact that
$\Cap^*(A,B)=\Cap(A,B)$ yields the second identity of the claim.
\end{proof}

Considering $\mc L^*$ in place of $\mc L$ we obtain from the previous
lemma that 
\begin{equation}
\label{2-08}
\Cap^*(A,B) \;=\; 
\oint_{\partial A} \left(a^\dagger\, \nabla h^* 
- h^*\, c\right)\cdot n_A \,d\mu_A 
\end{equation}

\smallskip\noindent{\bf Variational formulae for the capacity.}  Let
$\mc F$ be the Hilbert space of vector fields $\varphi : \Omega\to \bb
R^d$ endowed with the scalar product $\langle \cdot,\cdot \rangle$
given by:
\begin{equation*}
\langle \varphi, \psi \rangle \;:=\; 
\int_{\Omega} \varphi \cdot a_s^{-1} \, \psi\, d\mu \;. 
\end{equation*}
Let $\mc F_\gamma$, $\gamma\in \bb R$, be the closure in $\mc F$ of
the space of smooth vector fields $\varphi\in \mc F$ such that
\begin{equation}
\label{e:fc}
\nabla \cdot (e^{-V}\varphi)=0\;,\qquad 
\oint_{\partial A} \varphi \cdot n_A \, d\mu_A\,  \;=\; -\, \gamma\;. 
\end{equation}

Let $\mc C_{\alpha,\beta}$, $\alpha$, $\beta\in\bb R$, be the space of
smooth functions $f:\Omega \to \bb R$ such that $f \equiv \alpha$ on
$A$ and $f\equiv \beta$ on $B$. For $f \in \mc C_{\alpha,\beta}$
define
\begin{equation*}
\Psi_f \; :=\; a_s \, \nabla f \;, \quad 
\Phi_f \; :=\; a^\dagger \, \nabla f \;-\; f\, c\;.
\end{equation*}
Note that
\begin{equation}
\label{2-06}
\< \Psi_{h} , \Psi_{h}\> \;=\; \int_{\Omega} \nabla h \cdot a_s\,
\nabla h \, d\mu \;=\; \Cap(A,B)\;.
\end{equation}

\begin{lemma}
\label{l2-03}
For every $\varphi\in \mc F_\gamma$ and $f\in \mc C_{\alpha,0}$ we
have that
\begin{equation*}
\langle \Phi_f-\varphi \,,\, \Psi_h\rangle
\;=\; \gamma\;+\; \alpha\,\Cap(A,B)\;.
\end{equation*}
\end{lemma}

\begin{proof}
Fix $\varphi\in \mc F_\gamma$ and $f\in \mc C_{\alpha,0}$.
By definition of $\Phi_f$, 
\begin{equation*}
\langle \Phi_f-\varphi, \Psi_h\rangle \;=\;
\int_{\Omega} \left(a^\dagger \, \nabla f
\,-\, f\,c \,-\, \varphi\right)\cdot \nabla h\, d\mu \;.
\end{equation*}
Writing $a^\dagger \, \nabla f \cdot \nabla h$ as $\nabla f \cdot a
\nabla h$, and integrating by parts, since $f=\alpha$ on $\partial A$
and $f=0$ on $\partial B$, the previous term becomes
\begin{equation*}
- \int_{\Omega}  \Big( f\,e^{V} \nabla \cdot 
\left(e^{-V} a\, \nabla h\right) \,+\,  f\,c + \varphi \Big) \cdot \nabla h 
\, d\mu 
\;+\; \alpha \oint_{\partial A} a\, \nabla h \cdot n_A \, d\mu_A \;. 
\end{equation*}
By definition, the last integral is the capacity between $A$ and $B$,
while the expression involving $f$ is equal to $- f\, \mc L h$. This
expression vanishes because $h$ is $\mc L$-harmonic in
$\Omega$. Hence, since $h = \chi_A$ on $\partial A \cup \partial B$,
by an integration by part, the previous expression is equal
\begin{equation*}
\int_{\Omega} h \, e^V \nabla \cdot (e^{-V}\varphi) \, d\mu  
\;-\; \oint_{\partial A} \, \varphi \cdot  n_A \, d\mu_A  
\;+\; \alpha\,\Cap(A,B) \;.
\end{equation*}
By \eqref{e:fc}, this expression is equal to $\gamma +
\alpha\,\Cap(A,B)$, as claimed.
\end{proof}

\begin{theorem}[Dirichlet principle]
\label{th2-1}
Fix two disjoint subsets $A$, $B$ of $\bb T^d$ satisfying
\eqref{2-15}. Then,
\begin{equation*}
\Cap(A,B)\;=\; \inf_{f\in \mc C_{1,0}}\inf_{\varphi \in \mc F_0} 
\langle \Phi_f-\varphi,\Phi_f-\varphi\rangle\;,
\end{equation*}
and the infimum is attained for $h_\star = (1/2) (h+h^*)$ and
$\varphi_\star = \Phi_{h_\star}- \Psi_h$\, .
\end{theorem}

\begin{proof}
Fix $f\in \mc C_{1,0}$ and $\varphi \in \mc F_0$.  By
Lemma~\ref{l2-03}, applied with $\gamma=0$ and $\alpha=1$, and by
Schwarz inequality,
\begin{equation*}
\Cap(A,B)^2 \; =\; \langle \Phi_f-\varphi \,,\, \Psi_h\rangle^2 
\;\le\; \langle \Phi_f-\varphi \,,\, \Phi_f-\varphi\rangle \,
\langle \Psi_h, \Psi_h\rangle  \;.
\end{equation*}
By \eqref{2-06}, the last term is equal to $\Cap(A,B)$, so that
$\Cap(A,B)\le \langle \Phi_f-\varphi,\Phi_f-\varphi\rangle$ for every
$f$ in $\mc C_{1,0}$ and $\varphi$ in $\mc F_0$.

Recall from the statement of the theorem the definition of $h_\star$
and $\varphi_\star$. Since $\Phi_{h_\star} - \varphi_\star =
\Psi_{h}$, by \eqref{2-06}, $\Cap(A,B)=\langle \Phi_{h_\star} -
\varphi_\star \,,\, \Phi_{h_\star} - \varphi_\star\rangle$. Therefore,
to complete the proof of the theorem, it remains to check that
$h_\star$ belongs to $\mc C_{1,0}$, and $\varphi_\star$ to $\mc
F_0$. It is easy to check the first condition. For the second one,
observe that
\begin{equation*}
\nabla \cdot (e^{-V} \varphi_\star) \;=\;
\frac 12 \, e^{-V}\, \left( \mc L^* h^* -\mc L h\right)
\;-\; \frac 12 \, (h + h^*)\, \nabla \cdot (e^{-V} c) \;.
\end{equation*}
This expression vanishes on $\Omega$ by the harmonicity of $h$, $h^*$
and in view of \eqref{2-07}. On the other hand,
\begin{gather*}
\oint_{\partial A} \varphi_\star \cdot n_A \, d\mu_A \;=\;
\frac 12 \,  \oint_{\partial A} (a^\dagger \, \nabla h^* - h^* c) \cdot n_A \, d\mu_A 
\;-\; \frac 12 \, 
\oint_{\partial A} (a\, \nabla h + h \, c) \cdot n_A \, d\mu_A \;. 
\end{gather*}
By Lemma \ref{l2-01} and identity \eqref{2-08}, the previous
expression is equal to $(1/2) \{\Cap^*(A,B)-\Cap(A,B)\}= 0$, which
completes the proof of the theorem.
\end{proof}

\begin{theorem}[Thomson principle]
\label{th2-2}
Fix two disjoint subsets $A$, $B$ of $\bb T^d$ satisfying
\eqref{2-15}. Then,
\begin{equation*}
\frac 1{\Cap(A,B)}\;=\; \inf_{\varphi \in \mc F_1} \inf_{f\in \mc C_{0,0}} \,
\langle \Phi_f-\varphi \,,\, \Phi_f-\varphi\rangle\;.
\end{equation*}
Moreover, the infimum is attained at $h_\star=(h-h^*)/2\, \Cap(A,B)$
and $\varphi_\star=\Phi_{h_\star}- \Psi_{h/\Cap(A,B)}$.
\end{theorem}

\begin{proof}
Fix $\varphi$ in $\mc F_1$ and $f$ in $\mc C_{0,0}$.  By
Lemma~\ref{l2-03} (applied with $\alpha=0$ and $\gamma=1$), by
Schwarz inequality, and by \eqref{2-06},
\begin{equation*}
1\,=\, \langle \Phi_f-\varphi \,,\, \Psi_h\rangle^2 
\;\le\; \langle \Phi_f-\varphi \,,\, \Phi_f-\varphi\rangle \,
\langle \Psi_h \,,\,  \Psi_h\rangle
\,=\, \langle \Phi_f-\varphi \,,\, \Phi_f-\varphi\rangle \,
\Cap(A,B)\, .
\end{equation*}
By definition of $h_\star$, $\varphi_\star$,
$\Phi_{h_\star}-\varphi_\star = \Psi_{h/\Cap(A,B)}$, so that by
\eqref{2-06}, 
\begin{equation*}
\langle \Phi_{h_\star}-\varphi_\star,\Phi_{h_\star}- \varphi_\star\rangle  
\;=\;  \< \Psi_{h/\Cap(A,B)} \,,\, \Psi_{h/\Cap(A,B)} \> 
\;=\; 1/\Cap(A,B)\;.
\end{equation*}
It remains to check that $h_\star\in \mc C_{0,0}$, and $\varphi_\star
\in \mc F_1$. It is easy to verify the first condition. For the second
one, observe that
\begin{equation*}
\varphi_\star \;=\; \frac{-\, 1}{2\, \Cap (A,B)}\, \Big\{ \big[\, a\,
\nabla h \,+\, h\, c \, \big] + \big[\, a^\dagger \, \nabla h^* \,-\, h^*\,
c\, \big]\, \Big\}\; .
\end{equation*}
Therefore, $2\, \Cap (A,B) \, \nabla \cdot (e^{-V} \varphi_\star) = -
\, e^{-V} [\mc L h +\mc L^* h^*]=0$ on $\Omega$. Moreover,
\begin{align*}
& -\, 2\, \Cap(A,B) \, \oint_{\partial A} \varphi_\star \cdot n_A\, d\mu_A
\\
&\quad \;=\; 
\oint_{\partial A} (\, a\, \nabla h + h c \,) \cdot n_A \, d\mu_A \;+\; 
\oint_{\partial A} (\, a^\dagger \,\nabla h^* - h^* c \,) \cdot n_A \, d\mu_A 
\;. 
\end{align*}
By Lemma \ref{l2-01} and \eqref{2-08}, the right-hand side is equal to
$\Cap(A,B)+\Cap^*(A,B)= 2 \, \Cap(A,B)$. This proves that
$\varphi_\star$ belongs to $\mc F_1$, and completes the proof of the
theorem.
\end{proof}

\smallskip\noindent{\bf Reversible case.} In the reversible case,
$c=0$, $a$ symmetric, the optimal flow $\varphi$ in the Dirichlet
principle is the null one, so that
\begin{equation}
\label{2-9}
\Cap(A,B)\;=\; \inf_{f\in \mc C_{1,0}}
\langle \Phi_f ,\Phi_f \rangle \;=\; \inf_{f\in \mc C_{1,0}}
\int_\Omega \nabla f \cdot a\, \nabla f \, d\mu\;.
\end{equation}
In the last identity we used the fact that $\Phi_f = \Psi_f = a\nabla
f$. We thus recover the Dirichlet principle in the reversible context.

Similarly, in the reversible case, the optimal function $f$ in the
Thomson principle is the null one, so that
\begin{equation*}
\frac 1{\Cap(A,B)} \;=\; \inf_{\varphi \in \mc F_1} \,
\langle \varphi \,,\, \varphi\rangle\;,
\end{equation*}
which is the Thomson's principle in the reversible case.

We conclude this subsection comparing the capacity induced by the
generator $\mc L$ with the one induced by the symmetric part of the
generator, $\mc L^s$ given by \eqref{2-8}.

Fix two disjoint subsets $A$, $B$ satisfying \eqref{2-15}.  Denote by
$\Cap_s(A,B)$, the capacity between $A$ and $B$ induced $\mc
L^s$. Since $h$ belongs to $\mc C_{1,0}$, by \eqref{2-02} and
\eqref{2-9}, 
\begin{equation*}
\Cap_s (A,B) \;\le\; \Cap(A,B)\;.
\end{equation*}

In the case of Markov chains taking value in a countable state-space,
it is proved in Lemma 2.6 of \cite{GL} that if the generator satisfies
a sector condition with constant $C_0$,
\begin{equation*}
\Big( \int  (\mc L f) \, g \, d\mu \Big)^2 \;\le\; C_0 \,
\int  (-\, \mc L f) \, f \, d\mu \, \int  (-\, \mc L g) \, g \, d\mu\;, 
\end{equation*}
for all smooth functions $f$, $g$, then $\Cap (A,B) \le C_0\,
\Cap_s(A,B)$.

\smallskip\noindent{\bf Stochastic representation.} The operators $\mc
L$ and $\mc L^*$ are generators of Markov processes on $\bb T^d$ with
invariant measure $\mu$. More precisely, $\mc L$ is the generator of
the solution of the stochastic differential equation
\begin{equation}
\label{2-13}
dX_{t}\;=\; \big\{ -\, a^\dagger\, \nabla V \,+\, \nabla\cdot a \,+\,
c\} (X_t) \, dt \,+\, \sqrt{2\, a_s} \, dW_{t}\;,
\end{equation}
where $W_{t}$ is a standard $d$-dimensional Brownian motion, $\sqrt{2
  a_s}$ represents the symmetric, positive-definite square root of $2
a_s$, and $\nabla\cdot a$ is the vector field whose $j$-th coordinate
is $(\nabla \cdot a)_j = \sum_{1\le i\le d} \partial_{x_i} a_{i,j}$.
For $\mc L^*$, one has to replace the drift in \eqref{2-13} by $-\,
a\, \nabla V \,+\, \nabla\cdot a^\dagger \,-\, c$.

Denote by $C([0,+\infty);\bb T^d)$ the space of continuous functions
$X : [0,+\infty) \to \bb T^d$ endowed with the topology of locally
uniform convergence. Let $\{\bb P_x : x\in \bb T^d\}$, resp.  $\{\bb
P^*_x : x\in \bb T^d\}$, be the probability measures on
$C([0,+\infty);\bb T^d)$ induced by the Markov process associated to
the generator $\mc L$, resp. $\mc L^*$, starting from $x$.

Denote by $H_C$, $C$ a closed subset of $\bb T^d$, the hitting time of
$C$:
\begin{equation*}
H_C \;=\; \inf\{t\ge 0 : X_t \in C\}\;. 
\end{equation*}

\begin{lemma}
\label{l2-02}
Let $C$ be the closure of an open set with smooth boundary. Consider
two continuous functions $b$, $f:\bb T^d\to \bb R$. Let $u:\bb T^d\to
\bb R$ be given by
\begin{equation*}
u(x) \;:=\; \bb E_x\Big[ \, b(X(H_C)) \;+\; \int_0^{H_C} f(X(t))\,dt\, \Big]\;.
\end{equation*}
Then, $u$ is the unique solution to
\begin{equation}
\label{2-10}
\begin{cases}
\mc Lu= -\, f & \text{on $\bb T^d\setminus C$,} \\
u=b & \text{on $\partial C$.}
\end{cases}
\end{equation}
\end{lemma}

This result provides a stochastic representation for the harmonic
functions $h=h_{A,B}$, $h^* = h^*_{A,B}$ introduced previously:
\begin{equation*}
h(x) \;=\; \bb P_x [H_A<H_B] \;, \quad 
h^*(x) \;=\; \bb P^*_x [H_A<H_B] \;.
\end{equation*}

\smallskip\noindent{\bf Harmonic measure.} In view of the definition
\eqref{2-09} of the capacity, define the probability measure
$\nu\equiv \nu_{A,B}$ as the harmonic measure on $\partial A
\cup \partial B$ conditioned to $\partial A$ as
\begin{equation*}
d\nu: \;=\; \frac{1}{\Cap(A,B)} a^\dagger \, \nabla h^* \cdot n_A \, d\mu_A\;.
\end{equation*}

\begin{proposition}
\label{p2-1}
For each continuous function $f:\bb T^d \to \bb R$,
\begin{equation}
\label{2-11}
\bb E_{\nu}\Big[ \int_0^{H_B}f(X_t)\, dt\Big] \;=\;
\frac{1}{\Cap(A,B)}\int h^*\,f\,d\mu\;.
\end{equation}
\end{proposition}

\begin{proof}
Fix a continuous function $f$, and let $\Omega_B = \bb T^d \setminus
B$. Denote by $u$ the unique solution of the elliptic equation
\eqref{2-10} with $C=B$, $b\equiv 0$.  In view of Lemma \ref{l2-02}
and by definition of the harmonic measure $\nu$, the left hand
side of \eqref{2-11} is equal to
\begin{equation*}
\frac {1}{\Cap (A,B)} \oint_{\partial A} u \, [a^\dagger\, 
\nabla h^* ] \cdot n_{A} \, d\mu_A  \;.
\end{equation*}
The integral of the same expression at $\partial B$ vanishes because
$u$ vanishes on $\partial B$. Hence, by the divergence theorem, this
expression is equal to
\begin{equation*}
\frac {1} {\Cap (A,B)} \int_{\Omega} e^{V} \, \nabla \cdot 
\big\{e^{-V} \, [a^\dagger \, \nabla h^* ] \, u
\big\} \, d\mu  \;.
\end{equation*}
Since the equilibrium potential $h^*$ is harmonic on $\Omega$,
the previous equation is equal to
\begin{align*}
\frac {1} {\Cap (A,B)} \int_{\Omega} 
\nabla h^*  \, a \, \nabla u \, d\mu 
\;+\; \frac {1} {\Cap (A,B)} \int_{\Omega} 
u \, c\cdot \nabla h^* \, d\mu \;.
\end{align*}
Consider the first term. Since $\nabla h^*$ vanishes on $A$, we may
extend the integral to $\Omega_B = \bb T^d \setminus B$.  By the
divergence theorem and since the equilibrium potential $h^*$ vanishes
on $\partial B$, this expression is equal to
\begin{equation*}
- \; \frac {1} {\Cap (A,B)} \int_{\Omega_B} 
h^*\,  e^V\, \nabla \cdot \big\{ e^{-V} \, a \, 
\nabla u \big\} \, d\mu \;.
\end{equation*}
As $\mc L u = -f$ on $\Omega_B$, this expression is equal to
\begin{equation*}
\frac {1} {\Cap (A,B)} \int_{\Omega_B} h^* \, c \cdot \nabla u\, d\mu 
\;+\; \frac {1} {\Cap (A,B)} \int_{\Omega_B} h^*\, f \, d\mu \;.
\end{equation*}
Since the equilibrium potential $h^*$ vanishes on $B$, we may replace
$\Omega_B$ by $\bb T^d$ in the last integral.

Up to this point we proved that the left-hand side of \eqref{2-11} is
equal to
\begin{equation*}
\frac {1} {\Cap (A,B)} \Big\{ \int_{\bb T^d } h^*\, f \, d\mu
\;+\; \int_{\Omega_B} h^* \, c \cdot \nabla u\, d\mu 
\;+\; \int_{\Omega} u \, c\cdot \nabla h^* \, d\mu \Big\} \;. 
\end{equation*}
Since $\nabla h^* $ vanishes on $A\cup B$ and $h^*$ on $B$, we may
extend the last two integrals to $\bb T^d$. By \eqref{2-01}, the sum
of the last two terms vanishes, which completes the proof of the
proposition.
\end{proof}

Proposition \ref{p2-1} is due to Bovier, Eckhoff, Gayrard and Klein
\cite{BEGK01} for reversible Markov chains. A generalization to
non-reversible chains can be found in \cite{BL7}.  A Dirichlet
principle, as a double variational formula of type $\inf_f \sup_g$
involving functions, was proved by Pinsky \cite{p1, p2, p95} in the
context of diffusions. It has been derived by Doyle \cite{Do94} and by
Gaudilli\`ere and L. for Markov chains \cite{GL}. The Dirichlet
principle, stated in Theorem \ref{th2-1}, appeared in \cite{GL} for
Markov chains and is due to L., Mariani and Seo \cite{lms17} in the
context of diffusion processes. The Thomson principle, stated in
Theorem \ref{th2-2}, is due to Slowik \cite{Slo} in the context of
Markov chains and appeared in \cite{lms17} for diffusions.

\section{Recurrence of Markov chains}
\label{sec04}

The capacity is an effective tool to prove the recurrence or
transience of Markov processes whose stationary state are explicitly
known.

Consider the following open problem. Let $\bs X = \{(X_k,Y_k) : k\in
\bb Z\}$, be a sequence of independent, identically distributed random
variables such that $P[(X_0,Y_0)=(\pm 1, \pm 1)] = 1/4$ for all $4$
combinations of signs.  Given a random environment $\bs X$ consider
the discrete-time random walk on $\bb Z^2$ whose jump probabilities
are given by
\begin{equation}
\label{4-1}
p\big( (j,k) \,,\, (j+Y_k,k) \big) \;=\; 
p\big( (j,k) \,,\, (j,k+X_j) \big) \;=\; 1/2 \;\text{ for all }
(j,k) \in\bb Z^2\;.
\end{equation}
Denote by $Z_t = (Z^1_t, Z^2_t)$ the position at time $t\in \bb Z$ of
the random walk. Equation \eqref{4-1} states that in the horizontal
line $\{(p,q) : q=k\}$ $Z$ only jumps from $(j,k)$ to $(j+Y_k,k)$ for
all $j$. In other words, on each horizontal line the random walk is
totally asymmetric, but the direction of the jumps may be differ from
line to line. Similarly, on the vertical lines $\{(p,q) : p=j\}$ the
random walk is totally asymmetric and only jumps from $(j,k)$ to
$(j,k+X_j)$. It is not known if this random walk is recurrent or not
[almost surely with respect to the random environment].

Fix an environment $\bs X$, and let $\bs P^{\bs X}_{(j,k)}$ be the
distribution of the random walk $Z$ which starts at time $t=0$ from
$(j,k)$. Denote by $H^+_0$ the return time to $0$: $H^+_0 = \inf\{
t\ge 1 : Z_t = 0\}$. The random walk is recurrent if and only if $\bs
P^{\bs X}_{0} [H^+_0 = \infty]=0$. Let $\{B_N : N\ge 1\}$ be an increasing
sequence of finite sets such that $\cup_N B_N = \bb Z^2$, and note
that
\begin{equation}
\label{4-2}
\bs P^{\bs X}_{0} [H^+_0 = \infty] \;=\; \lim_{N\to\infty}
\bs P^{\bs X}_{0} [H_{B^c_N} < H^+_0] \;,
\end{equation}
where $H_{B^c_N}$ stands for the hitting time of $B^c_N$. 

In the context of discrete-time Markov chains evolving on a countable
state-space the capacity between two disjoint sets $A$, $B$ is given
by
\begin{equation*}
\Cap (A,B) \;=\; \sum_{x\in A}  M(x)\, P_{x} [H_B < H^+_A] \;,
\end{equation*}
where $M$ represents the stationary state of the Markov chain and
$H_B$, resp. $H^+_A$, the hitting time of the set $B$, resp. the
return time to the set $A$.

By the previous identity, the right-hand side of \eqref{4-2} can be
rewritten as
\begin{equation*}
\frac 1{ M_{\bs X}(0)} \, \lim_{N\to\infty} \Cap_{\bs X} \big(\{0\}\,,\, B^c_N
\big) \;,
\end{equation*}
where $M_{\bs X}$ represents the stationary state of the random
walk. It is easy to show that $M_{\bs X}$ does not depend on the
environment and is constant, $M_{\bs X}(z) =1$ for all $z\in \bb Z^2$.

In view of the Dirichlet principle, to prove that the random walk is
recurrent, one needs to find a sequence of functions $f_N$ in $\mc
C_{1,0}$ and of vector fields $\varphi_N$ in $\mc F_0$ [with $A=\{0\}$
$B=B^c_N$ and depending on the environment $\bs X$] for which
$\<\Phi_{f_N} - \varphi_N, \Phi_{f_N} - \varphi_N\>$ vanishes
asymptotically.

This has not been achieved yet. However, this is the simplest way to
prove that the symmetric, nearest-neighbor random walk on $\bb Z^2$ is
recurrent [$p((j,k),(j,k\pm 1)) = p((j,k),(j\pm 1,k)) = 1/4$]. In this
case also $M(z) =1$ for all $z\in \bb Z^2$. Consider $B_N = \{-(N-1),
\dots, N-1\}^2$, and set $\varphi_N=0$, $f_N(x) = 1 - \log |x|_m/\log
N$, $x\in B_N$, where $|0|_m =1$, $|(j,k)| = \max \{|j|,|k|\}$. For
these sequences,
\begin{equation*}
\<\Phi_{f_N} - \varphi_N, \Phi_{f_N} - \varphi_N\>  \;=\;
\frac 14 \sum_{j=1}^2 \sum_{x\in \bb Z^2} [f_N(x+e_j) - f_N(x)]^2 \;\le\;
\frac{C_0}{\log N}\;,
\end{equation*}
where $\{e_1, e_2\}$ stands for the canonical basis of $\bb R^2$ and
$C_0$ for a finite constant independent of $N$. This proves that the
$2$-dimensional, nearest-neighbor, symmetric random walk is recurrent.

\section{Eyring-Kramers formula for the transition time}
\label{sec05}

We examine in this section the stochastic differential equation
\eqref{2-13} as a small perturbation of a dynamical system $\dot x (t)
= F(x(t))$, by introducing a small parameter $\epsilon>0$ in the
equation.

To reduce the noise in \eqref{2-13}, we substitute $\sqrt{2\, a_s}$ in
the second term of the right-hand side by $\sqrt{2\, \epsilon\, a_s}$.
At this point, to keep the structure of the equation, we have to
replace in the first term $a$ by $\epsilon\, a$. To avoid the term
$-\, a^\dagger \nabla V$ to become small, we change $V$ to
$V/\epsilon$. After these modifications the equation \eqref{2-13}
becomes
\begin{equation}
\label{3-2}
dX_{t}^\epsilon\;=\;  
\big\{ -\, a^\dagger\, \nabla V \,+\, \epsilon \, \nabla\cdot a \,+\, c\,
\} \, (X ^\epsilon_t) \, dt
\,+\, \sqrt{2\, \epsilon\, a_s} \, dW_{t}\;.
\end{equation}
The diffusion $X_{t}^\epsilon$ is a small perturbation of the
dynamical system $\dot x (t) = -\, [a^\dagger \, \nabla V] (x(t)) +
c(x(t))$. For the equilibrium points of this ODE to be the critical
points of $V$, we require $V$ to be a Lyapounov functional.  This is
the case if $c \cdot \nabla V = 0$ on $\bb T^d$.

The generator of the diffusion $X_{t}^\epsilon$, denoted by $\mc
L_\epsilon$, is given by
\begin{equation*}
\mc{L}_{\epsilon}f \;=\;\epsilon\, e^{V/\epsilon} \, \nabla\,\cdot\,
\big\{ e^{-V/\epsilon} \,a\,  \nabla f \big\} \;+\; c\cdot \nabla f\;.
\end{equation*}

Let $\mu_\epsilon$ be the probability measure given by
\begin{equation}
\label{3-1}
\mu_\epsilon(dx) \;=\; \frac 1{Z_{\epsilon}}\, \exp\{-V(x)/\epsilon\}\, dx\;,
\end{equation}
where $Z_{\epsilon}$ is the normalizing constant,
$Z_{\epsilon}\;:=\;\int_{\bb T^d}\exp\{-V(x)/\epsilon\}\, dx$.  We
have seen in the previous section that $\mu_\epsilon$ is the
stationary state of the process $X_{t}^{\epsilon}$ provided $\nabla
\cdot (e^{-V/\epsilon} c) =0$. Since $c \cdot \nabla V$ vanishes, this
equation becomes $\nabla \cdot c=0$. We assume therefore that
\begin{equation}
\label{3-6}
c \cdot \nabla V \;=\; 0 \; \text{ and }\; 
\nabla \cdot c \;=\; 0 \; \text{ on }\; \bb T^d\;.
\end{equation}

We examine the transition time in the case where $V$ is a double well
potential. Assume that there exists an open set $\ms G\subset \bb T^d$
such that

\begin{itemize}
\item[(H1)] The potential $V$ has a finite number of critical points
  in $\ms G$.  Exactly two of them, denoted by $\bs{m}_{1}$ and
  $\bs{m}_{2}$, are local minima.  The Hessian of $V$ at each of these
  minima has $d$ strictly positive eigenvalues.

\item[(H2)] There is one and only one saddle point between
  $\bs{m}_{1}$ and $\bs{m}_{2}$ in $\ms G$, denoted by
  $\bs{\sigma}$. The Hessian of $V$ at $\bs \sigma$ has exactly one
  strictly negative eigenvalue and $(d-1)$ strictly positive
  eigenvalues.

\item[(H3)] We have that $V(\bs{\sigma}) < \inf_{x\in \partial\, \ms
    G} V(x)$.
\end{itemize}

Assume without loss of generality that $V(\bs{m}_{2}) \le
V(\bs{m}_{1})$, so that $\bs{m}_{2}$ is the global minimum of the
potential $V$ in $\ms G$. Denote by $\Omega$ the level set of the
potential defined by saddle point, $\Omega = \{ x \in \ms G : V(x)<
V(\bs{\sigma})\}$. Let $\ms {V}_{1}$, $\ms {V}_{2}$ be two domains
with smooth boundary containing $\bs{m}_1$ and $\bs{m}_2$,
respectively, and contained in $\Omega$:
\begin{equation}
\label{4-3}
\bs m_i \;\in\;
\ms V_i \;\subset\; \{x \in \ms G :  V(x) < V(\bs{\sigma})- \kappa \}
\end{equation}
for some $\kappa>0$.

Denote by $\nabla^2 V (x)$ the Hessian of $V$ at $x$. By Lemma 10.1
of \cite{LS1}, both $[\nabla^2 V \, a](\bs\sigma)$ and $[\nabla^2
V \, a^\dagger](\bs\sigma)$ have a unique (and the same) negative
eigenvalue.  Denote by $-\mu$ this common negative eigenvalue.

Let $\bb P^\epsilon_x$, $x\in \bb T^d$, be the probability measure on
$C(\bb R_+, \bb T^d)$ induced by the Markov process $X ^{\epsilon}_t$
starting from $x$. Expectation with respect to $\bb P^\epsilon_x$ is
represented by $\bb E^\epsilon_x$.

\begin{theorem}[Eyring-Kramers formula]
\label{thmp2}
We have that
\begin{equation}
\label{3-5}
\bb{E}^\epsilon_{\bs{m}_{1}} \left[H_{\ms {V}_{2}}\right]\;=\;
\left[1+o_{\epsilon}(1)\right]\, \mf p
\, e^{\Lambda /\epsilon} \;, \quad \text{\rm where }\;
\mf p\, \;=\; \frac{2\pi}{\mu} \,
\frac {\sqrt{-\det\left[\nabla^2 V \, (\bs{\sigma})\right]}}
{\sqrt{\det\left[(\nabla^2 V) (\bs{m}_{1})\right]}}
\end{equation}
and $\Lambda = V(\bs{\sigma}) - V(\bs{m}_1)$.
\end{theorem}

The term $\mf p$ is called the prefactor. It can be understood as the
first-order term in the expansion in $\epsilon$ of the exponential
barrier. Let $\bb{E}^\epsilon_{\bs{m}_{1}} [H_{\ms {V}_{2}}] = \exp\{
\Lambda(\epsilon) / \epsilon\}$. Theorem \ref{thmp2} states that
$\Lambda(\epsilon) = \Lambda + \epsilon \log \mf p + o(\epsilon)$.

The proof of this theorem in the case $c=0$ and $a$ independent of
$x$, $a(x)=a$, can be found in \cite{lms17}. Uniqueness of local
minima and of saddle points connecting the wells is not required
there. The same argument should apply to the general case under the
hypotheses \eqref{19}, \eqref{3-6}, but the proof has not been
written.

The $0$-th order term in the expansion, $\Lambda$, can be obtained
from Freidlin and Wentzell large deviations theory of random
perturbations of dynamical systems \cite{fw98}. The pre-factor $\mf p$
has been calculated rigorously for reversible diffusions by Sugiura
\cite{s95, s01} [based on asymptotics of the principal eigenvalue and
eigenfunction for a Dirichlet boundary value problem in a bounded
domain], and independently, by Bovier, Eckhoff, Gayrard and Klein
\cite{BEGK1} [based on potential theory]. We refer to \cite{b13} for a
recent review.

In the context of chemical reactions, the transition time
$\bb{E}^\epsilon_{\bs{m}_{1}} [H_{\ms {V}_{2}}]$ corresponds to the
inverse of the rate of a reaction. The so-called ``Arrenhuis law''
relates the rate of a reaction to the absolute temperature. It seems
to have been first discovered empirically by Hood \cite{h78}.  van't
Hoff \cite{hoff96} proposed a thermodynamical derivation of the law,
and Arrhenius \cite{a89} physical arguments based on molecular
dynamics. In the self-adjoint case, the pre-factor $\mf p$ first
appeared in Eyring \cite{e35} and in more explicit form in Kramers
\cite{k40}. Bouchet and Reygnier \cite{BR} derived the formula in the
non-reversible situation.

% prova

\section{Metastability}
\label{sec03}

We developed in these last years a robust method to prove the
metastable behavior of Markov processes based on potential theory.  We
report in this section recent developments which rely on asymptotic
properties of elliptic operators.

We first define metastability. Let $Z_\epsilon (t)$ be a sequence of
Markov processes taking values in some space $E_\epsilon$. Let $\{\ms
E^1_\epsilon, \dots, \ms E^n_\epsilon, \Delta_\epsilon\}$ be a
partition of the set $E_\epsilon$, and set $\ms E_\epsilon = \ms
E^1_\epsilon \cup \cdots \cup \ms E^n_\epsilon$.

Fix a sequence of positive numbers $\theta_\epsilon$, and denote by
$\widehat Z_\epsilon(t)$ the process $Z_\epsilon(t)$ speeded-up by
$\theta_\epsilon$: $\widehat Z_\epsilon(t) = Z_\epsilon(t\,
\theta_\epsilon)$. Denote by $\widehat {\bs P}_{\epsilon,x}$, $x\in
E_\epsilon$, the distribution of the process $\widehat Z_\epsilon(t)$
starting from $x$.  Let $S =\{1, \dots, n\}$, $S_0 =\{ 0 \} \cup S$,
and let $\Upsilon_\epsilon : E_\epsilon \to S_0$ be the projection
given by
\begin{equation}
\label{5-6}
\Upsilon_\epsilon (x) \;=\; \sum_{j=1}^{n} j \, \chi_{\ms E^j_\epsilon} (x) \;.
\end{equation}
Note that points in $\Delta_\epsilon$ are mapped to $0$. Denote by
$z_\epsilon(t)$ the $S_0$-valued process defined by
\begin{equation*}
z_\epsilon(t) \;=\; \Upsilon_\epsilon (\widehat Z_\epsilon(t)) 
\;=\; \Upsilon_\epsilon (Z_\epsilon(t \theta_\epsilon)) \;.
\end{equation*}
The process $z_\epsilon(t)$ is usually not Markovian.

\begin{definition}
  \label{d1} {\rm [Metastability]}. We say that the process
  $Z_\epsilon (t)$ is metastable in the time scale $\theta_\epsilon$,
  with metastable sets $\ms E^1_\epsilon, \dots, \ms E^n_\epsilon$ if
  there exists a $S$-valued, continuous-time Markov chain $z (t)$ such
  that for all $x\in \ms E_\epsilon$ the $\widehat {\bs
    P}_{\epsilon, x}$-finite-dimensional distributions of $
  z_\epsilon(t)$ converge to the finite-dimensional distributions of
  $z (t)$.
\end{definition}

The Markov chain $z (t)$ is called the {\color{blue} reduced chain}.
Mind that the reduced chain does not take the value $0$. The sojourns
of $\widehat Z_\epsilon(t)$ at $\Delta_\epsilon$ are washed-out in the
limit. Of course, the same process $Z_\epsilon(t)$ may exhibit
different metastable behaviors in different time-scales or even
different metastable behaviors in the same time-scale but in different
regions of the space, inaccessible one to the other in that
time-scale.

In some examples \cite{jlt1, jlt2, bcl} the set $S$ may be countably
infinite. In these cases $\Upsilon_\epsilon$ is a projection from
$E_\epsilon$ to a finite set $S_\epsilon \cup \{0\}$, where
$S_\epsilon$ increases to a countable set $S$, and we require $\#
E_\epsilon/\# S_\epsilon \to 0$.  \smallskip

In the remaining part of this section we prove that under certain
hypotheses the diffusion $X^\epsilon_t$ is metastable. Some of these
conditions are not needed, but they simplify the presentation. The
reader will find in the references finer results.

We assume from now on that the potential $V$ fulfills the following
set of assumptions. There exists an open set $\ms G$ of $\bb T^d$ such
that

\begin{itemize}
\item[(H1')] The function $V$ has a finite number of critical points
  in $\ms G$.  The global minima of $V$ are represented by
  $m_{1}, \dots, m_n$. They all belong to $\ms G$ and they
  are all at the same height: $V(m_i)=V(m_j)$ for all $i$,
  $j$. The Hessian of $V$ at each of these minima has $d$ strictly
  positive eigenvalues.

\item[(H2')] Denote by $\{\sigma_1, \dots, \sigma_\ell\}$ the
  set of saddle points between the global minima. Assume that all
  saddle points are at the same height and that the Hessian of $V$ at
  these points has exactly one strictly negative eigenvalue and
  $(d-1)$ strictly positive eigenvalues.

\item[(H3')] We have that $V(\sigma_1) < \inf_{x\in \partial\, \ms
    G} V(x)$.  
\end{itemize}

Denote by $\Omega$ the level set of the potential defined by the
height of the saddle points: $\Omega = \{ x \in \ms G : V(x) <
V(\sigma_1)\}$. Let $\ms W_1, \dots, \ms W_p$ be the connected
components of $\Omega$. Assume that each of these sets contains one
and only one global minima, so that $p=n$. Denote by $\ms V_1, \dots,
\ms V_n$ domains with smooth boundaries satisfying \eqref{4-3} for
$1\le i\le n$, and let
\begin{equation}
\label{65}
\ms V \;=\; \bigcup_{j=1}^{n} \ms V_j \;, \quad \Delta \;=\; \bb T^d
\,\setminus\, \ms V\;, \quad \breve{\ms V}_j \;=\; \bigcup_{k:k\not =
  j} \ms V_k\;. 
\end{equation}

Recall from \eqref{3-5} the definition of $\Lambda$.  Denote by
$\widehat X^\epsilon_t$ the process $X^\epsilon_t$ speeded-up by
{\color{blue} $\theta_\epsilon = e^{\Lambda/\epsilon}$}. This is the
diffusion on $\bb T^d$ whose generator, denoted by $\widehat {\mc
  L}_\epsilon$, is given by $\color{blue}\widehat {\mc L}_\epsilon =
\theta_\epsilon\, {\mc L}_\epsilon$.  Denote by $\color{blue} \bb
P^\epsilon_x$, resp. $\color{blue} \widehat {\bb P}^\epsilon_x$, $x\in
\bb T^d$, the probability measure on $C(\bb R_+, \bb T^d)$ induced by
the diffusion $X^\epsilon_t$, resp. $\widehat X^\epsilon_t$, starting
from $x$. Expectation with respect to ${\bb P}^\epsilon_x$,
is represented by ${\bb E}^\epsilon_x$.

Let $\color{blue} S =\{ 1, \dots, n\}$, $\color{blue} S_0 = \{0\} \cup
S$. Denote by $\Upsilon: \bb T^d \to S_0$ the projection given by
\eqref{5-6} with $\ms E^j_\epsilon$ replaced by $\ms V_j$, and let
$x_\epsilon(t)$ be the $S_0$-valued process defined by
\begin{equation*}
x_\epsilon(t) \;=\; \Upsilon (\widehat X^\epsilon_t) 
\;=\; \Upsilon (X^\epsilon(t \theta_\epsilon)) \;.
\end{equation*}
Note that $x_\epsilon(t)$ is not Markovian.

The proof of the metastable behavior of the diffusion $X^\epsilon_t$
is divided in four steps. We first show that in the time scale
$\theta_\epsilon$ the process $X^\epsilon_t$ spends a negligible
amount of time in the set $\Delta$. Then, we derive a candidate for
the $S$-valued Markov chain which is supposed to describe the
asymptotic behavior of the process among the wells.  In the third
step, we prove that the projection of the trace of
$\widehat{X}^\epsilon_t$ on $\ms V$ converges to the $S$-valued Markov
chain introduced in the second step. Finally, we show that the
previous results together with an extra condition yield the
convergence of the finite-dimensional distributions of
$x_\epsilon(t)$.

\smallskip
\noindent{\bf Step 1: The set $\Delta$ is negligible.} We first examine
in the next lemma the time spent on the set $\Delta$.

\begin{lemma}
\label{l5-1}
For all $t>0$,
\begin{equation}
\label{5-1}
\lim_{\epsilon \to 0} \, \sup_{x\in \ms V} \, \bb E^\epsilon_{x} \Big[  
\int_{0}^{t} \chi_{\Delta}(X(s \theta_\epsilon))\, ds \, \Big] \;=\; 0 \;.
\end{equation}
\end{lemma}

\begin{proof}
Here is a sketch of the proof of this result which highlights the
relevance of the variational formulae for the capacity. Denote by
$\color{blue} \Cap_\epsilon(\ms A, \ms B)$ the capacity between two
disjoint subsets $\ms A$, $\ms B$ with respect to the diffusion
$X^\epsilon_t$.  

Fix $1\le j\le n$ and assume that $x$ belongs to $\ms
V_j$. The time scale $\theta_\epsilon$ is of the order of the
transition time $H_{\breve{\ms V}_j}$, where the $\breve{\ms V}_j$ has
been introduced in \eqref{65}. The expectation appearing in the
statement of the lemma is therefore of the same order of
\begin{equation*}
\frac 1{\theta_\epsilon} \, \bb E^\epsilon_{x} \Big[  
\int_{0}^{H_{\breve{\ms V}_j}} 
\chi_{\Delta}(X(s))\, ds \, \Big] \;\sim\;
\frac 1{\theta_\epsilon \Cap_\epsilon(\ms V_j, \breve{\ms V}_j)} 
\int_{\bb T^d} \chi_\Delta\, h_{\ms V_j, \breve{\ms V}_j} 
\, d\mu_\epsilon\;,
\end{equation*}
where last step follows from Proposition \ref{p2-1}. It would be an
identity if we had the harmonic measure in place of the Dirac measure
concentrated on $x$, but these expectations should not be very
different because $x$ belongs to the basin of attraction of
$m_j$. Since $\mu_\epsilon(\Delta) \to 0$, the proof is completed if
we can show, using the variational principles, that $\theta_\epsilon
\Cap_\epsilon(\ms V_j, \breve{\ms V}_j)$ converges to a positive value.
\end{proof}

\noindent{\bf Step 2: The reduced chain.} The time-scale
$\theta_\epsilon$ at which the process $X^\epsilon_t$ evolves among
the wells should be of the order of the transition time $\bb
E^\epsilon_{m_j} [\, H(\breve{\ms V}_j) \,]$. Hence, by Proposition
\ref{p2-1}, 
\begin{equation*}
\theta_\epsilon \;\sim\; 
\bb E^\epsilon_{m_j} \big [\, H(\breve{\ms V}_j) \,\big] \;\sim\; 
\frac{1}{\Cap_\epsilon(\ms V_j \,,\, \breve{\ms V}_j)}\int 
h_{\ms V_j, \breve{\ms V}_j^*} \,d\mu\;.
\end{equation*}
Since $m_j$ is a global minimum of $V$, the last integral is of order
$1$ because the harmonic function $h_{\ms V_j, \breve{\ms V}_j^*}$ is
equal to $1$ at $\ms V_j$. We conclude that the time-scale
$\theta_\epsilon$ should be of the order $\Cap_\epsilon(\ms V_j \,,\,
\breve{\ms V}_j)^{-1}$.

It is proved in \cite{BL1, BL7}, in the context of Markov chains
taking values in a countable state space, that under certain
assumptions
\begin{equation*}
\lambda_j \;:=\; \lim_{\epsilon \to 0}
\theta_\epsilon \, \frac 1{\mu_\epsilon (\ms V_j)}
\, \Cap_\epsilon(\ms V_j \,,\, \breve{\ms V}_j) 
\end{equation*}
represents the holding time at $j$ of the reduced chain. Moreover, in
the reversible case, the jump rates $r(j,k)$ of the reduced chain are
given by 
\begin{equation*}
r(j,k) \;=\; \lim_{\epsilon \to 0} \frac 1{2\, \mu_\epsilon (\ms V_j)}
\Big\{ \Cap_\epsilon (\ms V_j \,,\, \breve{\ms V}_j) 
+ \Cap_\epsilon(\ms V_k \,,\, \breve{\ms V}_k) 
- \Cap_\epsilon (\ms V_j \cup \ms V_k \,,\, 
\ms V \setminus [\ms V_j \cup \ms V_k] )\, \Big\} \,. 
\end{equation*}
 
In the non-reversible case, the jump rates are more difficult to
derive. By \cite[Proposition 4.2]{BL7}, still in the context of Markov
chains taking values in a countable state space,
\begin{equation*}
r(j,k) \;=\; \lambda_j \, \lim_{\epsilon \to 0} \overline{\bb
P}^\epsilon_{m_j} \big[ \, H(\ms V_k) <  H\big(\ms V \setminus [\ms V_j
\cup \ms V_k] \big)\, \big]\;,
\end{equation*}
where $\overline{\bb P}^\epsilon_{m_j}$ represents the distribution of
the process in which the well $\ms V_j$ has been collapsed to the
point $m_j$. Estimates on the harmonic function appearing on the
right-hand of this equation are obtained by showing that this function
solves a variational problem, similar to the one for the capacity, and
then that to be optimal, a function has to take a precise value at the
set $\ms V_j$. We refer to \cite{l2014, LS1} for details, where this
program has been successfully undertaken for two different models.

Assume that one can compute the asymptotic jump rates through the
previous formulae or that one can guess by other means the jump rates
of the reduced chain.  Denote by $\bs L$ the generator of the
$S$-valued continuous-time Markov chain induced by these jump rates.
Let $\color{blue} D(\bb R_+, E)$, $E$ a metric space, be the space of
$E$-valued, right-continuous functions with left-limits endowed with
the Skorohod topology, and let $\color{blue} \bs Q_j$, $j\in S$, the
measure on $D(\bb R_+, S)$ induced by the Markov chain with generator
$\bs L$ starting from $j$.

\smallskip
\noindent{\bf Step 3: Convergence of the trace.}
We turn to the convergence of the trace process.  Recall that
$\widehat X^\epsilon_t$ represents the process $X^\epsilon_t$
speeded-up by $\theta_\epsilon$.  Denote by $T_{\ms V} (t)$, $t\ge 0$,
the total time spent by the diffusion $\widehat X^\epsilon$ on the set
$\ms V$ in the time interval $[0,t]$:
\begin{equation*}
T_{\ms V} (t)  \,:=\, \int_0^t \chi_{\ms V} (\widehat X^\epsilon_s) \,ds\;,
\end{equation*}
Denote by $\{S_{\ms V} (t) : t\ge 0\}$ the generalized inverse of
$T_{\ms V} (t)$:
\begin{equation*}
S_{\ms V} (t) \,:=\, \sup\{s\ge 0 : T_{\ms V} (s) \le t \}\,.
\end{equation*}
Clearly, for all $r\ge 0$, $t\ge 0$,
\begin{equation} 
\label{60}
\{ S_{\ms V} (r) \ge t\} \;=\; \{ T_{\ms V} (t) \le r\}\;.
\end{equation}
It is also clear that for any starting point $x\in \bb T^d$,
$\lim_{t\to \infty} T_{\ms V} (t) = \infty$ almost surely. Therefore,
the random path $\{Y_\epsilon (t) : t\ge 0\}$, given by $\color{blue}
Y_\epsilon(t) := \widehat X^\epsilon (S_{\ms V} (t))$, is well defined
for all $t\ge 0$ and takes value in the set $\ms V$.  We call the
process $Y_\epsilon(t)$ the {\color{blue}trace} of $\widehat
X^\epsilon_t$ on the set $\ms V$.

The process $Y_\epsilon(t)$ is Markovian provided the initial
filtration is large enough.  Indeed, denote by $\{\ms F^0_t : t\ge
0\}$ the natural filtration of $C(\bb R_+, \bb T^d)$: $\ms F^0_t =
\sigma(X_s: 0\le s\le t)$.  Fix $x_0\in \ms V$ and denote by $\{\ms
F_t : t\ge 0\}$ the usual augmentation of $\{\ms F^0_t : t\ge 0\}$
with respect to $\bb P^\epsilon_{x_0}$. We refer to Section III.9 of
\cite{rw94} for a precise definition, and to \cite{ls2017} for a proof of
the next result which relies on the identity \eqref{60}.

\begin{lemma}
\label{as15}
For each $t\ge 0$, $S_{\ms V} (t)$ is a stopping time with respect to
the filtration $\{\ms F_t\}$. 
\end{lemma}

As $S_{\ms V} (t)$ is a stopping time with respect to the filtration
$\{\ms F_t\}$, $Y_\epsilon(t)$ is a $\ms V$-valued, Markov process
with respect to the filtration $\color{blue} \ms G_t := \ms F_{S(t)}$.
Let $\Psi: \ms V \to S$ be the projection given by
\begin{equation*}
\Psi(x) \;=\; \sum_{j=1}^{n} j \, \chi_{\ms V_j} (x) \;,
\end{equation*}
and denote by $y_\epsilon(t)$ the $S$-valued process obtained by
projecting $Y_\epsilon(t)$ with $\Psi$:
\begin{equation*}
y_\epsilon(t) \;=\; \Psi(Y_\epsilon(t))\;.
\end{equation*}
Note that the process $y_\epsilon(t)$ is not Markovian.

Denote by $\bb Q^\epsilon_x$, resp. $\mb Q^\epsilon_x$, $x\in \ms V$,
the probability measure on $D(\bb R_+, \ms V)$, resp. $D(\bb R_+, S)$,
induced by the process $Y_\epsilon(t)$, resp. $y_\epsilon(t)$, given
that $Y_\epsilon(0) = x$.  Fix $j\in S$, $x\in \ms V_j$. As usual, the
proof that $\mb Q^\epsilon_x$ converges to $\mb Q_j$ is divided in two
steps. We first show that the sequence $\mb Q^\epsilon_x$ is tight and
then we prove the uniqueness of limit points.

\begin{lemma}
\label{l3-1}
Assume that conditions \eqref{5-1} is in force. Suppose, furthermore,
that
\begin{equation}
\label{3-4}
\lim_{r \to 0} \, \limsup_{\epsilon \to 0}\,
\max_{1\le j\le n} \, \sup_{x\in\ms V_j} \, \bb P^\epsilon_{x}
\big[\, H(\breve{\ms V}_j) \,\le\, r\, \theta_\epsilon \,\big]  \;=\; 0\;. 
\end{equation}
Then, for every $1\le j\le n$, $x_0\in \ms V_j$, the sequence of
measures $\mb Q^\epsilon_{x_0}$ is tight. Moreover, every limit point
$\mb Q^*$ of the sequence $\mb Q^\epsilon_{x_0}$ is such that
\begin{equation}
\label{3-3}
\mb Q^* \{ x : x(0) = j\} \;=\; 1 \quad \text{and} \quad 
\mb Q^* \{ x : x(t) \not = x(t-)\}\;=\; 0
\end{equation}
for every $t>0$.
\end{lemma}

A proof of this result for one-dimensional diffusions is presented in
\cite[Lemma 7.5]{ls2017}.  Condition \eqref{3-4} asserts that in the
time-scale $\theta_\epsilon$, the process $X^\epsilon_t$ may not jump
instantaneously from one well to the other.  We show in Section 8 of
this article that the probability $\bb P^\epsilon_{x} [\, H(\breve{\ms
  V}_j) \,\le\, r\, \theta_\epsilon \,]$ is bounded by the capacity
between two sets for an enlarged process. The proof of this lemma is
thus reduced to an estimate of capacities.

The proof of uniqueness relies on the characterization of
continuous-time Markov chains as solutions of martingale problems.
One needs to show that
\begin{equation}
\label{5-4}
\bs F(y(t)) \;-\; \int_0^t (\bs L \bs F)(y(s))\, ds
\end{equation}
is a martingale under $\mb Q$ for all functions $\bs F: S \to \bb R$
and all limit point $\mb Q$ of the sequence $\mb Q^\epsilon_{x_0}$.

We proved in \cite{BL1, BL7} that this property is in force in the
context of countable state spaces provided the mean jump rates
converge and if each well $\ms V_j$ has an element $z_j$ such that
\begin{equation}
\label{5-2}
\lim_{\epsilon \to 0} \sup_{y\in \ms V_j , y\not = z_j} 
\frac{\Cap_\epsilon(\ms V_j, \breve{\ms V}_j)}
{\Cap_\epsilon(\{y\}, \{z_j\})} \;=\; 0\;.
\end{equation}
The point $z_j$ is not special. Typically, if \eqref{5-2} holds for a
point $z_j$ in the well, it holds for all the other ones. We refer to
\cite{BL1, BL7} for details.

Condition \eqref{5-2} has been derived for Markov processes which
``visit points'', that is, for Markov processes which visit all points
of a well before reaching another well. This is the case of condensing
zero-range processes \cite{BL3, l2014}, random walks in potential
fields \cite{lmt2015, LS1}, one-dimensional diffusions \cite{ls2017},
and for all processes whose wells are reduced to singletons, as the
inclusion processes \cite{bdg2016}.

We present here an alternative method to deduce \eqref{5-4} which
relies on certain asymptotic properties of the elliptic operator $\ms
L_\epsilon$.  Fix a function $\bs F \colon S \to \bb R$, let $\bs G =
\bs L \bs F$, and let $g\colon \bb T^d\to\bb R$ be given by
\begin{equation*}
g \;=\; \sum_{i=1}^{n} \bs G(i) \, \chi_{\ms V_i}\;.  
\end{equation*}
Assume that there exists a sequence of function $g_\epsilon : \bb T^d
\to \bb R$ such that
\begin{itemize}
\item[(P1)] $g_\epsilon$ vanishes on $\ms V^c$ and converges to $g$
  uniformly on $\ms V$;
\item[(P2)] The Poisson equation $\widehat{\mc L}_\epsilon f = g_\epsilon$
  in $\bb T^d$ has a solution denoted by $f_\epsilon$. Moreover, there
  exists a finite constant $C_0$ such that
\begin{equation*}
\sup_{0<\epsilon<1}\, \sup_{x\in\bb T^d} |f_\epsilon(x)| \;\le\; C_0\;,
\quad\text{and} \quad
\lim_{\epsilon\to 0} 
\sup_{x\in \ms V} \big|\, f_\epsilon(x) 
- f(x)\, \big| \;=\; 0\;,
\end{equation*}
where $f: \bb T^d \to \bb R$ is given by $f = \sum_{1\le i \le n} \bs
F(i) \, \chi_{\ms V_i}$.
\end{itemize}

The natural candidate for $g_\epsilon$ in conditions (P1) and (P2) is
the function $g$ itself. However, as the process is ergodic, the
Poisson equation $\widehat{\mc L}_\epsilon f = b$ has a solution only if
$b$ has mean zero with respect to $\mu_{\epsilon}$.  We need therefore
to modify $g$ to obtain a mean-zero function.  Denote by $\pi$ the
stationary state of the Markov chain whose generator is $\bs L$. We
expect $\mu_\epsilon(\ms V_i)$ to converge to $\pi_i$. Hence,
\begin{equation*}
\lim_{\epsilon\to 0} E_{\mu_{\epsilon}}[g] \;=\; 
\lim_{\epsilon\to 0} \sum_{i=1}^{n} \bs G(i) \,
\mu_\epsilon(\ms V_i) \; = \; \sum_{i=1}^{n} \bs L \bs F (i) \,
\pi_i \;=\; 0\;.
\end{equation*}
A reasonable candidate for $g_\epsilon$ is thus $g \,-\, r(\epsilon)
\, \chi_{\ms V_1}$, where $r(\epsilon) =
E_{\mu_\epsilon}[g]/\mu_\epsilon(\ms V_1)$.

Properties (P1), (P2) have been proved in \cite{et2016, st2017} for
elliptic operators on $\bb R^d$ of the form $\mc L_\epsilon f =
e^{V/\epsilon} \nabla \cdot (e^{-V/\epsilon} a \nabla f)$ and in
\cite{ls2017} for one-dimensional diffusions with periodic boundary
conditions. It is an open problem to prove these conditions in the
context of interacting particle systems.

\begin{lemma}
\label{l08}
Fix $1\le j\le n$ and $x_0\in \ms V_j$. Assume that conditions
(P1) and (P2) are in force. Let $\mb Q^*$ be a limit point of the
sequence $\mb Q^\epsilon_{x_0}$ satisfying \eqref{3-3}. Then, for
every $\bs F: S \to \bb R$, \eqref{5-4} is a martingale under the
measure $\mb Q^*$.
\end{lemma}

\begin{proof}
Fix $1\le j\le n$, $x_0\in \ms V_j$ and a function $\bs F: S
\to \bb R$.  Let $f_\epsilon: \bb T^d \to \bb R$ be the function given
by assumption (P2). Then,
\begin{equation*}
M_\epsilon(t) \;=\; f_\epsilon (\widehat X^\epsilon_t) \;-\; \int_0^t 
(\widehat{\mc L}_\epsilon f_\epsilon)(\widehat X^\epsilon_s)\, ds \;=\;
f_\epsilon (\widehat X^\epsilon_t) \;-\; \int_0^t 
g_\epsilon (\widehat X^\epsilon_s)\, ds
\end{equation*}
is a martingale with respect to the filtration $\ms F_t$ and the
measure $\widehat{\bb P}^\epsilon_{x_0}$. Since $\{S_{\ms V}(t): t\ge
0\}$ are stopping times with respect to $\ms F_t$,
\begin{equation*}
\widehat M_\epsilon(t) \;=\; M_\epsilon(S_{\ms V}(t)) \;=\;  
f_\epsilon (Y_\epsilon(t)) \;-\; \int_0^{S_{\ms V}(t)} 
g_\epsilon (\widehat X^\epsilon_s)\, ds
\end{equation*}
is a martingale with respect to the filtration $\ms G_t$.  Since
$g_\epsilon$ vanishes on $\ms V^c$, by a change of variables,
\begin{equation*}
\int_0^{S_{\ms V}(t)} g_\epsilon (\widehat X^\epsilon_s)\, ds \;=\;
\int_0^{S_{\ms V}(t)} g_\epsilon (\widehat X^\epsilon_s)\, 
\chi_{\ms V} (\widehat X^\epsilon_s) \, ds \;=\;
\int_0^{t} g_\epsilon (\widehat X^\epsilon(S_{\ms V}(s)))\, ds\;.
\end{equation*}
Hence,
\begin{equation*}
\widehat M_\epsilon(t) \;=\; 
f_\epsilon (Y_\epsilon(t)) \;-\; 
\int_0^{t} g_\epsilon (Y_\epsilon (s))\, ds
\end{equation*}
is a $\{\ms G_t\}$-martingale under the measure $\bb
Q^\epsilon_{x_0}$.

By (P1) and (P2), $g_\epsilon$, resp. $f_\epsilon$, converge to $g$,
resp. $f$, uniformly in $\ms V$ as $\epsilon \to 0$. Hence, since
$Y_\epsilon(s) \in \ms V$ for all $s\ge 0$, we may replace in the
previous equation $g_\epsilon$, $f_\epsilon$ by $g$, $f$,
respectively, at a cost which vanishes as $\epsilon\to 0$. Therefore,
\begin{equation*}
\widehat M_\epsilon(t) \;=\; 
f (Y_\epsilon(t)) \;-\; \int_0^{t} g (Y_\epsilon (s))\, ds \;+\;
 o(1)\, 
\end{equation*}
is a $\{\ms G_t\}$-martingale under the measure $\bb
Q^\epsilon_{x_0}$.

Since $f$ and $g$ are constant on each set $\ms V_i$, $f
(Y_\epsilon(t)) = \bs F(y_\epsilon(t))$, $g (Y_\epsilon(t)) = \bs
G (y_\epsilon(t))$.  By the second condition in \eqref{3-3}, $\mb
Q^*$ is concentrated on trajectories which are continuous at any fixed
time with probability $1$. We may, therefore, pass to the limit and
conclude that $\bs F(y(t)) \;-\; \int_0^t (\bs L \bs F)(y(s))\, ds$ is
a martingale under $\mb Q^*$.
\end{proof}

\begin{theorem}
\label{th01}
Assume that conditions (P1), (P2), \eqref{5-1}, \eqref{3-4} are in
force.  Fix $j\in S$ and $x_0\in \ms V_j$. The sequence of measures
$\mb Q^\epsilon_{x_0}$ converges, as $\epsilon \to 0$, to the
probability measure $\mb Q_j$.
\end{theorem}

\begin{proof}
The assertion is a consequence of Lemma \ref{l3-1}, Lemma \ref{l08} and
the fact that there is only one measure $\mb Q$ on $D(\bb R_+, S)$
such that $\mb Q[x(0)=j]=1$ and such that \eqref{5-4} 
is a martingale for all $\bs F : S\to\bb R$.
\end{proof}

\noindent{\bf Step 4: The finite-dimensional distributions.}  By
\cite[Proposition 1.1]{llm2016}, the finite-dimen\-sional distributions
of $x_\epsilon(t)$ converge to the finite-dimensional distributions of
$y(t)$ if the process $y_\epsilon(t)$ converges in the Skorohod
topology to $y(t)$ [Theorem \ref{th01}], if in the time-scale
$\theta_\epsilon$ the total time spent in $\Delta$ is negligible
[Lemma \ref{l5-1}] and if
\begin{equation*}
\lim_{\delta\to 0} \limsup_{\epsilon\to 0} \sup_{x\in \ms V} 
\sup_{\delta\le s\le 2\delta} \bb P^\epsilon_{x} 
[ X^\epsilon(s\theta_\epsilon) \in \Delta ]\;=\;0\;. 
\end{equation*}

This completes the argument.  The convergence of the
finite-dimensional distributions of $X^\epsilon_t$ and sharp
asymptotics for the transition time in the context of diffusions were
first obtained by Sugiura \cite{s95, s01}. The approach presented in
this section to prove the metastable behavior of a Markov process has
been proposed by Beltr\'an and L. \cite{BL1, BL7}. It has been
successfully applied to many models quoted in this section.  For
further reading on metastability, we refer to the books by Olivieri
and Vares \cite{ov} and Bovier and den Hollander \cite{bh}.

\smallskip\noindent{\bf Acknowledgments.} The results presented in
this review are the outcome of long standing collaborations. The
author wishes to thank his colleagues and friends, J. Beltr\'an,
E. Chavez, A. Gaudilli\`ere, M. Jara, M. Loulakis, M. Mariani,
R. Misturini, M. Mourragui, I. Seo, A. Teixeira, K. Tsunoda.


\begin{thebibliography}{99}
\bibitem{a89} S. Arrhenius: On the reaction velocity of the inversion
  of cane sugar by acids.  J. Phys. Chem. {\bf 4}, 226
  (1889). Partially translated to english in ``Selected readings in
  chemical kinetics'' edited by M. Back and K. Laidler, Pergamon
  press, Oxford 1967. %%

\bibitem{bcl} J. Beltr\'an, E. Chavez, C. Landim: From coalescing
  random walks on a torus to Kingman's coalescent. preprint (2017). %%

\bibitem{BL1} J. Beltr\'{a}n, C. Landim: Tunneling and metastability
  of continuous time Markov chains. J. Stat. Phys. {\bf 140},
  1065-1114, (2010). %%

\bibitem{BL3} J. Beltr\'{a}n, C. Landim: Metastability of reversible
  condensed zero range processes on a finite
  set. Probab. Th. Rel. Fields {\bf 152}, 781-807 (2012). %%

\bibitem{BL7} J. Beltr\'{a}n, C. Landim: Tunneling and metastability
  of continuous time Markov chains II. J. Stat. Phys. {\bf 149},
  598-618, (2012). %%

\bibitem{b13} N. Berglund: Kramers' law : validity, derivations and
 generalisations. Markov Process. Related Fields, {\bf 19}, 459-490
 (2013). %%

\bibitem{bdg2016} A. Bianchi, S. Dommers, C. Giardin{\`a}:
  Metastability in the reversible inclusion process. Elect.
  J. Probab. {\bf 22}, 70 (2017). %%

\bibitem{BR} F. Bouchet, J. Reygner: Generalisation of the
  Eyring-Kramers transition rate formula to irreversible diffusion
  processes. preprint (2015) http://arxiv.org/abs/1507.02104 %%
  
\bibitem{BEGK01} A. Bovier, M. Eckhoff, V. Gayrard, M. Klein:
  Metastability in stochastic dynamics of disordered mean-field
  models. Probab. Theory Relat. Fields {\bf 119}, 99-161 (2001)

\bibitem{BEGK1} A. Bovier, M. Eckhoff, V. Gayrard, M. Klein:
  Metastability in reversible diffusion process I. Sharp asymptotics
  for capacities and exit times. J. Eur. Math. Soc. {\bf 6}, 399--424
  (2004). %%

\bibitem{bh} A. Bovier, F. den Hollander: {\sl Metastability: a
  potential-theoretic approach}. Grundlehren der mathematischen
  Wissenschaften 351, Springer, Berlin, 2015. %%

% \bibitem{Dir50}
%   P. Dirichlet. Abh. K\"oniglich. Preuss. Akad. Wiss. 99-116 (1850)

\bibitem{Do94} P. Doyle: Energy for Markov chains. Unpublished
  manuscript available at http://www.math.dartmouth.edu/\~doyle (1994).

\bibitem{et2016} L. C. Evans, P. R. Tabrizian: Asymptotic for scaled
  Kramers-Smoluchowski equations. SIAM J. Math. Anal.  {\bf 48},
  2944-2961 (2016) %%

\bibitem{e35} H. Eyring: The activated complex in chemical
  reactions. J. Chem. Phys. {\bf 3}, 107-115 (1935). %%

\bibitem{fw98} M. I. Freidlin, A. D. Wentzell: {\it Random
    perturbations of dynamical systems}. Second edition. Grundlehren
  der Mathematischen Wissenschaften [Fundamental Principles of
  Mathematical Sciences], {\bf 260}. Springer-Verlag, New York, 1998. %%

% \bibitem{f} A. Friedman. {\sl Stochastic differential equations and
%     applications}. Academic Press, 1975

\bibitem{GL} A. Gaudilli\`ere, C. Landim: A Dirichlet principle for
  non reversible Markov chains and some recurrence theorems.
  Probab. Theory Related Fields {\bf 158}, 55--89 (2014). %%

\bibitem{gt} D. Gilbarg, N. S. Trudinger. {\sl Elliptic partial
    differential equations of second order}. Springer, 2015. %%

% \bibitem{hkn} B. Helffer, M. Klein, F. Nier: Quantitative analysis of
%   metastability in reversible diffusion processes via a Witten complex
%   approach. Mat. Contemp., {\bf 26}, 41--86 (2004).

\bibitem{hoff96} J. van't Hoff: {\sl Studies in Chemical Dynamics}
  Chemical publishing Company, 1896. %%

\bibitem{h78} J. Hood. Phil. Mag. {\bf 6}, 371 (1878) and {\bf 20},
  323 (1885). %%

\bibitem{jlt1} M. Jara, C. Landim, A. Teixeira: Quenched scaling
  limits of trap models. Ann. Probab. {\bf 39}, 176-223 (2011).

\bibitem{jlt2} M. Jara, C. Landim, A. Teixeira: Universality of trap
  models in the ergodic time scale. Ann. Probab. {\bf 42},
  2497-2557 (2014). %%

\bibitem{Kel} O. D. Kellogg. {\sl Foundations of potential
    theory}. Dover Publications Inc., New York, 1954. %%
  
\bibitem{k40} H. A. Kramers: Brownian motion in a field of force and
  the diffusion model of chemical reactions. Physica {\bf 7}, 284-304,
  (1940). %%

\bibitem{l2014} C. Landim; Metastability for a non-reversible
  dynamics: the evolution of the condensate in totally asymmetric zero
  range processes. Commun. Math. Phys. {\bf 330}, 1--32 (2014). %%

\bibitem{llm2016} C. Landim, M. Loulakis, M. Mourragui: Metastable
  Markov chains.  arXiv:1703.09481 (2017). %%

\bibitem{lmt2015} C. Landim, R. Misturini, K. Tsunoda; Metastability
  of reversible random walks in potential fields. J. Stat. Phys. {\bf 160}
  1449--1482 (2015). %%

\bibitem{lms17} C. Landim, M. Mariani, I. Seo: A Dirichlet and a
  Thomson principle for non-selfadjoint elliptic operators,
  metastability in non-reversible diffusion processes.
  arXiv:1701.00985, (2017). %%

\bibitem{LS1} C. Landim, I. Seo: Metastability of non-reversible random
  walks in a potential field, the Eyring-Kramers transition rate
  formula. To appear in Commun. Pure Appl. Math. arXiv:1605.01009
  (2016). %%

% \bibitem{LS2} C. Landim, I. Seo: Metastability of non-reversible,
%   mean-field Potts model with three spins.  J. Stat. Phys.  {\bf 165},
%   693-726 (2016)

\bibitem{ls2017} C. Landim, I. Seo: Metastability of one-dimensional,
  non-reversible diffusions with periodic boundary conditions. 
   arXiv:1710.06672 (2017). %%

% \bibitem{n1687}  I. Newton: {\sl Philosophi\ae Naturalis Principia
%     Mathematica}, 1687. %%

\bibitem{ov} E. Olivieri, M. E. Vares: {\sl Large deviations and
    metastability}. In: Encyclopedia of Mathematics and its
  Applications, vol. 100. Cambridge University Press, Cambridge 2005. %%

\bibitem{p1} R. G. Pinsky: A generalized Dirichlet principle for
  second order nonselfadjoint elliptic operators. SIAM
  J. Math. Anal. {\bf 19}, 204-213 (1988).

\bibitem{p2} R. G. Pinsky: A minimax variational formula giving
  necessary and sufficient conditions for recurrence or transience of
  multidimensional diffusion processes. Ann. Probab. {\bf 16}, 662-671
  (1988).

\bibitem{p95} R. G. Pinsky. {\sl Positive harmonic functions and
    diffusion} Cambridge studies in advanced mathematics
  Vol. 45. Cambridge University Press, 1995.

\bibitem{rw94} L.  C.  G.  Rogers, D.  Williams: {\sl Diffusions,
    Markov Processes, and Martingales: Volume 1,
    Foundations}. Cambridge University Press, 1994. %%

\bibitem{st2017} I. Seo, P. R. Tabrizian: Asymptotics for scaled
  Kramers-Smoluchowski equations in several dimensions with general
  potentials. preprint (2017). %%

\bibitem{Slo} M. Slowik: A note on variational representations of
  capacities for reversible and nonreversible Markov
  chains. unpublished, Technische Universit\"{a}t Berlin, 2012. %%

% \bibitem{sv}  D. W. Stroock and S. R. S. Varadhan. {\sl
%     Multidimensional diffusion processes}. Springer, 2007

\bibitem{s95} M. Sugiura: Metastable behaviors of diffusion processes
  with small parameter. J. Math. Soc.  Japan {\bf 47}, 755-788
  (1995). %%

\bibitem{s01} M. Sugiura: Asymptotic behaviors on the small parameter
  exit problems and the singularly perturbation problems. Ryukyu
  Math. J. {\bf 14}, 79-118 (2001). %%


\end{thebibliography}
\end{document}